\documentclass[11pt]{article}

\usepackage[T1]{fontenc}
\usepackage{lmodern}
\usepackage[margin=1.12in]{geometry}
\usepackage{amsmath,amssymb,amsthm,mathtools}
\usepackage{microtype}
\usepackage[hidelinks]{hyperref}
\hypersetup{
  pdftitle={Counterexamples to the xz-Conjecture and the Mathieu Conjecture for SU(2)},
  pdfauthor={Christopher D. Long},
  pdfsubject={Explicit moment counterexamples for the xz-conjecture and SU(2)},
  pdfkeywords={Mathieu conjecture, SU(2), xz-conjecture, Laurent polynomial, Haar integral, Jacobian Conjecture}
}
\newcommand{\C}{\mathbb C}
\newcommand{\T}{\mathbb T}
\newcommand{\CT}{\operatorname{CT}}
\newcommand{\Sp}{\operatorname{Sp}}
\newcommand{\I}{\mathcal I}

\newtheorem{theorem}{Theorem}[section]
\newtheorem{proposition}[theorem]{Proposition}
\newtheorem{corollary}[theorem]{Corollary}
\newtheorem{lemma}[theorem]{Lemma}
\theoremstyle{remark}
\newtheorem{remark}[theorem]{Remark}

\title{Counterexamples to the $xz$-Conjecture and the Mathieu Conjecture for $SU(2)$}
\author{Christopher D. Long\\\texttt{galizur@gmail.com}}
\date{July 21, 2026}

\begin{document}
\maketitle

\begin{abstract}
Let
\[
  \I(h)=\int_0^1\int_{\T}h(x,z)\,\frac{dz}{2\pi iz}\,dx
  \qquad
  \bigl(h\in\C[x,z,z^{-1}]\bigr).
\]
We give the three-term Laurent polynomial
\[
  f(x,z)=(1-z^{-1})\bigl((1-x)+xz\bigr)
\]
for which
\[
  \I(f^n)=0,
  \qquad
  \I(z^{-1}f^n)=\frac{(-1)^{n-1}}{n+1}\neq0
  \qquad(n\geq1).
\]
Since $\Sp(f)=\{-1,0,1\}$, this disproves the $xz$-conjecture already with one interval variable and one torus variable, and it also shows that $\ker\I$ is not a Mathieu--Zhao subspace.  Padding gives counterexamples to every mixed case of the $xz$-conjecture.  Writing the coordinate functions on $SU(2)$ as
\[
  g=\begin{pmatrix}a&c\\ b&d\end{pmatrix},
\]
the same example lifts, through the integration formula of M\"uger and Tuset, to the regular functions
\[
  F=(1+c)(ad+b),\qquad G=-c,
\]
which satisfy
\[
  \int_{SU(2)}F^n\,dg=0,
  \qquad
  \int_{SU(2)}F^nG\,dg=\frac{(-1)^{n-1}}{n+1}\neq0
\]
for every $n\geq1$.  Thus the Mathieu conjecture for $SU(2)$ is false.
\end{abstract}



\section{Introduction}

Let $K$ be a compact connected Lie group with normalized Haar measure $dg$.  Mathieu conjectured that if $f$ and $g$ are complex-valued $K$-finite functions and
\[
  \int_K f(k)^n\,dk=0\qquad(n\geq1),
\]
then
\[
  \int_K f(k)^n g(k)\,dk=0
\]
for all sufficiently large $n$ \cite{Mathieu}.  Duistermaat and van der Kallen proved the conjecture for compact connected abelian groups \cite{DvK}.  Dings and Koelink studied the first nonabelian case $K=SU(2)$, proving the conjecture for several restricted classes and reducing the general problem to a support conjecture for matrix coefficients \cite{DingsKoelink}.

M\"uger and Tuset subsequently gave an abelian reduction of the $SU(2)$ problem \cite{MugerTuset}.  For integers $k,l\geq0$, an admissible function is a polynomial
\[
  h\in
  \C[x_1,\dots,x_l,z_1,z_1^{-1},\dots,z_k,z_k^{-1}],
\]
written as
\[
  h(x,z)=\sum_{m\in\mathbb Z^k}c_m(x)z^m,
  \qquad c_m\in\C[x_1,\dots,x_l].
\]
Its torus spectrum is
\[
  \Sp(h)=\{m\in\mathbb Z^k:c_m\neq0\}.
\]
The $xz$-conjecture asserts that
\begin{equation}\label{eq:xz-conjecture}
  \int_{[0,1]^l}\int_{\T^k}h(x,z)^n\,d^\times z\,dx=0
  \quad(n\geq1)
  \quad\Longrightarrow\quad
  0\notin\operatorname{conv}\Sp(h),
\end{equation}
where each torus is equipped with normalized Haar measure.  The cases $l=0$ and $k=0$ were known, while M\"uger and Tuset recorded that, to the best of their knowledge, every genuinely mixed case was open, including $k=l=1$ \cite[Remark~2.4]{MugerTuset}.  They also proved that the case $(k,l)=(2,1)$ would imply the Mathieu conjecture for $SU(2)$ \cite[Theorem~2.5]{MugerTuset}.

Mathieu proved that his conjecture for all compact connected Lie groups implies Keller's Jacobian Conjecture \cite{Mathieu,Keller}.  On July 20, 2026, Levent Alp\"oge announced an explicit polynomial map in three variables with constant nonzero Jacobian determinant and a fiber containing three distinct points \cite{Alpoge}.  The determinant and collision identities are direct algebraic checks, so the map disproves the Jacobian Conjecture in dimension three and, by adjoining identity coordinates, in every dimension at least three.  It follows that Mathieu's conjecture cannot hold for all compact connected Lie groups.  This argument does not determine whether the conjecture holds for $SU(2)$, nor does it decide the separate $xz(1,1)$ problem.  No part of the proofs below depends on the Jacobian counterexample.  The examples below settle both questions directly.

Our basic Laurent polynomial is
\begin{equation}\label{eq:basic-f}
  f(x,z)
  =(1-z^{-1})\bigl((1-x)+xz\bigr)
  =xz+(1-2x)-(1-x)z^{-1}.
\end{equation}
Its spectrum is the three-point affine circuit
\[
  \Sp(f)=\{-1,0,1\},
\]
so $0\in\operatorname{conv}\Sp(f)$.  Nevertheless, all pure moments vanish.  More strongly, the fixed multiplier $z^{-1}$ gives a nonzero mixed moment for every positive exponent.

The mechanism is the elementary beta-binomial identity
\begin{equation}\label{eq:beta-binomial-intro}
  \binom nk\int_0^1x^k(1-x)^{n-k}\,dx=\frac1{n+1}.
\end{equation}
It makes every Bernstein stratum contribute with the same normalized weight; the remaining coefficients form an alternating binomial row and cancel exactly.

\section{A counterexample to \texorpdfstring{$xz(1,1)$}{xz(1,1)}}

Equip $\T=\{z\in\C:|z|=1\}$ with normalized Haar measure
\[
  d^\times z=\frac{dz}{2\pi iz}.
\]
For $h\in\C[x,z,z^{-1}]$, put
\begin{equation}\label{eq:I-def}
  \I(h)
  =\int_0^1\int_{\T}h(x,z)\,d^\times z\,dx
  =\int_0^1\CT_z h(x,z)\,dx.
\end{equation}

\begin{theorem}\label{thm:basic-xz}
Let $f$ be the polynomial in \eqref{eq:basic-f}.  Then, for every integer $n\geq1$,
\begin{equation}\label{eq:basic-moments}
  \I(f^n)=0,
  \qquad
  \I(z^{-1}f^n)=\frac{(-1)^{n-1}}{n+1}.
\end{equation}
Consequently, the $xz$-conjecture is false for $(k,l)=(1,1)$, and $\ker\I$ is not a Mathieu--Zhao subspace of $\C[x,z,z^{-1}]$.
\end{theorem}

\begin{proof}
By the binomial theorem,
\[
  \bigl((1-x)+xz\bigr)^n
  =\sum_{k=0}^n\binom nk x^k(1-x)^{n-k}z^k.
\]
The beta integral gives
\[
  \int_0^1x^k(1-x)^{n-k}\,dx
  =\frac{k!(n-k)!}{(n+1)!}
  =\frac{1}{(n+1)\binom nk}.
\]
Therefore
\begin{equation}\label{eq:bernstein-average}
  \int_0^1\bigl((1-x)+xz\bigr)^n\,dx
  =\frac1{n+1}\sum_{k=0}^n z^k.
\end{equation}
Using \eqref{eq:basic-f}, \eqref{eq:I-def}, and \eqref{eq:bernstein-average}, we obtain
\begin{align*}
  \I(f^n)
  &=\frac1{n+1}\CT_z
    \left[
      (1-z^{-1})^n\sum_{k=0}^n z^k
    \right]\\
  &=\frac1{n+1}\CT_z
    \left[
      \left(\sum_{j=0}^n(-1)^j\binom nj z^{-j}\right)
      \left(\sum_{k=0}^n z^k\right)
    \right].
\end{align*}
A constant term occurs exactly when $j=k$, so
\[
  \I(f^n)
  =\frac1{n+1}\sum_{j=0}^n(-1)^j\binom nj
  =\frac{(1-1)^n}{n+1}=0.
\]

For the mixed moment, the constant-term condition in
\[
  z^{-1}(1-z^{-1})^n\sum_{k=0}^n z^k
\]
is $-1-j+k=0$, or $k=j+1$.  Hence
\begin{align*}
  \I(z^{-1}f^n)
  &=\frac1{n+1}\sum_{j=0}^{n-1}(-1)^j\binom nj\\
  &=\frac{(-1)^{n-1}}{n+1},
\end{align*}
where the last equality follows by subtracting the final term from
$\sum_{j=0}^n(-1)^j\binom nj=0$.

Since $0\in\operatorname{conv}\Sp(f)$ while all pure moments vanish, \eqref{eq:xz-conjecture} fails for $(1,1)$.  Finally, $f^n\in\ker\I$ for all $n\geq1$, but $z^{-1}f^n\notin\ker\I$ for every $n\geq1$.  This is exactly the failure of the Mathieu--Zhao property.
\end{proof}

\begin{corollary}\label{cor:all-mixed-xz}
The $xz$-conjecture is false for every pair $(k,l)$ with $k\geq1$ and $l\geq1$.  For every such pair, the kernel of the corresponding product-integration functional is not a Mathieu--Zhao subspace.
\end{corollary}

\begin{proof}
For larger $k$ and $l$, use the same polynomial $f(x_1,z_1)$ and ignore the remaining variables.  The product integrals reduce to those in Theorem~\ref{thm:basic-xz}, and the multiplier $z_1^{-1}$ gives the same failure of eventual mixed-moment vanishing.
\end{proof}

\begin{remark}[The generating function is constant]\label{rem:generating}
For a Laurent trinomial $az^{-1}+b+cz$, one has the formal identity
\[
  \CT_z\frac{1}{1-t(az^{-1}+b+cz)}
  =\frac{1}{\sqrt{(1-bt)^2-4act^2}},
\]
with the square root having constant term $1$.  For the polynomial $f$ in \eqref{eq:basic-f},
\[
  a=-(1-x),\qquad b=1-2x,\qquad c=x,
\]
and the radicand becomes
\[
  (1-t(1-2x))^2+4t^2x(1-x)
  =(1-t)^2+4tx.
\]
Consequently, as a formal power series, or analytically for $|t|$ sufficiently small,
\begin{align*}
  1+\sum_{n\geq1}\I(f^n)t^n
  &=\int_0^1\frac{dx}{\sqrt{(1-t)^2+4tx}}\\
  &=\frac{\sqrt{(1+t)^2}-\sqrt{(1-t)^2}}{2t}=1.
\end{align*}
Thus the moment generating function is identically constant; the vanishing is not merely asymptotic or sporadic.
\end{remark}

\section{A family of circuit counterexamples}

The preceding example belongs to a simple family supported on symmetric three-point circuits.

\begin{proposition}\label{prop:family}
Let $d\in\mathbb Z_{\geq1}$ and let $\lambda,\mu\in\C^\times$.  Define
\begin{equation}\label{eq:family}
  f_{d,\lambda,\mu}(x,z)
  =\lambda(1-\mu z^{-d})
   \bigl((1-x)+x\mu^{-1}z^d\bigr).
\end{equation}
Then
\[
  \Sp(f_{d,\lambda,\mu})=\{-d,0,d\},
\]
and, for every $n\geq1$,
\begin{align}
  \I(f_{d,\lambda,\mu}^n)&=0,
  \label{eq:family-pure}\\
  \I(z^{-d}f_{d,\lambda,\mu}^n)
  &=\frac{(-1)^{n-1}\lambda^n\mu^{-1}}{n+1}.
  \label{eq:family-mixed}
\end{align}
\end{proposition}

\begin{proof}
Expanding \eqref{eq:family} gives
\[
  f_{d,\lambda,\mu}
  =\lambda\mu^{-1}xz^d
   +\lambda(1-2x)
   -\lambda\mu(1-x)z^{-d},
\]
so its spectrum is $\{-d,0,d\}$.  The same beta calculation as in
\eqref{eq:bernstein-average} gives
\[
  \int_0^1\bigl((1-x)+x\mu^{-1}z^d\bigr)^n\,dx
  =\frac1{n+1}\sum_{k=0}^n\mu^{-k}z^{dk}.
\]
Multiplication by
\[
  \lambda^n(1-\mu z^{-d})^n
  =\lambda^n\sum_{j=0}^n(-1)^j\binom nj\mu^jz^{-dj}
\]
shows that the pure constant term is
\[
  \frac{\lambda^n}{n+1}
  \sum_{j=0}^n(-1)^j\binom nj=0.
\]
After multiplication by $z^{-d}$, the constant-term condition is again $k=j+1$, and the powers of $\mu$ contribute
$\mu^j\mu^{-(j+1)}=\mu^{-1}$.  Thus
\[
  \I(z^{-d}f_{d,\lambda,\mu}^n)
  =\frac{\lambda^n\mu^{-1}}{n+1}
   \sum_{j=0}^{n-1}(-1)^j\binom nj,
\]
which is \eqref{eq:family-mixed}.
\end{proof}

\section{The lift to \texorpdfstring{$SU(2)$}{SU(2)}}

Write an element of $SU(2)$ and its coordinate functions as
\[
  g=\begin{pmatrix}a&c\\ b&d\end{pmatrix}.
\]
Thus $c=-\overline b$, $d=\overline a$, and $ad-bc=1$ on the group.
A regular function on $SU(2)$ is the restriction of a polynomial in
$\C[a,b,c,d]$; such functions are $SU(2)$-finite.  We use normalized Haar measure $dg$.

M\"uger and Tuset prove the following integration formula \cite[Lemma~2.1]{MugerTuset}.  For completeness, we include a monomial verification.  Define
\begin{equation}\label{eq:beta-map}
  \beta(z_1,z_2,x)
  =\bigl((1-x)z_2,\;xz_1,\;-z_1^{-1},\;z_2^{-1}\bigr).
\end{equation}

\begin{lemma}[The $SU(2)$ integration formula]\label{lem:su2-integration}
For every $P\in\C[a,b,c,d]$,
\begin{equation}\label{eq:su2-integration}
  \int_{SU(2)}P(a,b,c,d)\,dg
  =\int_0^1\int_{\T^2}
    P\bigl(\beta(z_1,z_2,x)\bigr)
    \frac{dz_1}{2\pi iz_1}\frac{dz_2}{2\pi iz_2}\,dx.
\end{equation}
\end{lemma}

\begin{proof}
By linearity it is enough to take
\[
  P=a^r b^s c^t d^u,
  \qquad r,s,t,u\in\mathbb Z_{\geq0}.
\]
The standard polar description of normalized Haar measure on $SU(2)$ writes
\[
  a=\sqrt{1-x}\,z_2,
  \quad b=\sqrt{x}\,z_1,
  \quad c=-\sqrt{x}\,z_1^{-1},
  \quad d=\sqrt{1-x}\,z_2^{-1},
\]
where $x$ is uniformly distributed on $[0,1]$ and $z_1,z_2$ are independently Haar distributed on $\T$.  The two torus integrals therefore force $r=u$ and $s=t$, and the remaining beta integral gives
\begin{equation}\label{eq:su2-monomial}
  \int_{SU(2)}a^r b^s c^t d^u\,dg
  =(-1)^s\delta_{r,u}\delta_{s,t}
    \frac{r!\,s!}{(r+s+1)!}.
\end{equation}
On the right-hand side of \eqref{eq:su2-integration}, the same monomial becomes
\[
  (-1)^t(1-x)^r x^s z_1^{s-t}z_2^{r-u}.
\]
The torus integrals again force $r=u$ and $s=t$, after which
\[
  \int_0^1(1-x)^r x^s\,dx
  =\frac{r!\,s!}{(r+s+1)!}.
\]
Thus both sides agree with \eqref{eq:su2-monomial}, proving the formula.
\end{proof}

\begin{theorem}\label{thm:su2}
Let $\lambda,\mu\in\C^\times$ and define regular functions on $SU(2)$ by
\begin{equation}\label{eq:FG-family}
  F_{\lambda,\mu}
  =\lambda(1+\mu c)(ad+\mu^{-1}b),
  \qquad
  G=-c.
\end{equation}
Then, for every integer $n\geq1$,
\begin{align}
  \int_{SU(2)}F_{\lambda,\mu}(g)^n\,dg&=0,
  \label{eq:su2-pure}\\
  \int_{SU(2)}F_{\lambda,\mu}(g)^nG(g)\,dg
  &=\frac{(-1)^{n-1}\lambda^n\mu^{-1}}{n+1}\neq0.
  \label{eq:su2-mixed}
\end{align}
In particular, the Mathieu conjecture for $SU(2)$ is false.
\end{theorem}

\begin{proof}
Under the substitution \eqref{eq:beta-map},
\[
  1+\mu c\longmapsto1-\mu z_1^{-1},
  \qquad
  ad+\mu^{-1}b
  \longmapsto(1-x)+x\mu^{-1}z_1,
\]
and
\[
  G=-c\longmapsto z_1^{-1}.
\]
Thus $F_{\lambda,\mu}$ maps precisely to the function
$f_{1,\lambda,\mu}(x,z_1)$ of Proposition~\ref{prop:family}, while $G$ maps to $z_1^{-1}$.  These expressions are independent of $z_2$, so integration over the second torus contributes the factor $1$.  Equations \eqref{eq:su2-pure} and \eqref{eq:su2-mixed} now follow from \eqref{eq:su2-integration}, \eqref{eq:family-pure}, and \eqref{eq:family-mixed}.

The hypothesis in Mathieu's conjecture holds for $F_{\lambda,\mu}$, but its predicted conclusion fails for the fixed regular function $G$ at every positive exponent.  Hence the conjecture is false for $SU(2)$.
\end{proof}

\begin{remark}\label{rem:one-way-jacobian}
Mathieu's implication from his compact-group conjecture to the Jacobian Conjecture is one-way.  Consequently, Theorem~\ref{thm:su2} does not by itself imply that the two-dimensional Jacobian Conjecture is false.
\end{remark}

Taking $\lambda=\mu=1$ gives the particularly small pair announced in the abstract:
\[
  F=(1+c)(ad+b),
  \qquad
  G=-c.
\]
The failure of eventual vanishing is maximal in the sense that the mixed moment is nonzero for every $n\geq1$.

\section{The cancellation mechanism}

The counterexamples arise from an exact interaction between the binomial expansion and the beta integral.  For each $0\leq k\leq n$,
\[
  \binom nk
  \int_0^1x^k(1-x)^{n-k}\,dx
  =\frac1{n+1}.
\]
Thus integrating the Bernstein expansion of $((1-x)+xz)^n$ removes the binomial coefficient completely and leaves the geometric sum
\[
  \frac1{n+1}(1+z+\cdots+z^n).
\]
The factor $(1-z^{-1})^n$ then contributes the alternating binomial row.  The pure moment is the full sum $(1-1)^n$, while multiplication by $z^{-1}$ removes the final term and leaves $(-1)^{n-1}$.  This explains both the all-moment cancellation and the persistent mixed moments without any limiting argument.

The examples were found during work on the Gaussian Moments Conjecture of Derksen, van den Essen, and Zhao \cite{DVEZ}.  The contrast is instructive: Gaussian radial moments carry factorial weights, whereas the beta integral in \eqref{eq:beta-binomial-intro} exactly cancels the multinomial coefficient and permits the alternating cancellation above.  No result concerning Gaussian moments is needed for the proofs in this paper.

\section*{AI provenance and author responsibility}

The $xz(1,1)$ counterexample and its lift to $SU(2)$ were discovered by ChatGPT 5.6 Sol Pro during an interactive research session with the author exploring the Gaussian Moments Conjecture.  The system also supplied the beta-binomial proof, identified the family in Proposition~\ref{prop:family}, and assisted with drafting and algebraic checks.  The author independently verified the displayed identities and bears full responsibility for the mathematics, the exposition, and every claim in this manuscript.

\end{document}